\documentclass[oneside,10pt]{amsart}          
\usepackage{amsfonts,amsmath,latexsym,amssymb} 
\usepackage{graphicx}
\usepackage{epsfig,graphicx}
\usepackage{subfigure}

\def\dom{D}

\def\NN{{\mathbb{N}}}
\def\RR{\mathbb{R}}
\def\CC{\mathbb{C}}

\def\dd{\:\mathrm{d}}
\def\ee{\mathrm{e}}

\newtheorem{theorem}{Theorem}[section]

\newtheorem{lemma}[theorem]{Lemma}
\newtheorem{proposition}[theorem]{Proposition}
\theoremstyle{definition}

\newtheorem{assumptions}[theorem]{Assumptions}

\begin{document}

\title[PDE approximation of ODE systems]{PDE approximation of large systems of differential equations}

\author[B\'atkai]{Andr\'as B\'atkai, \'Agnes Havasi, R\'obert Horv\'ath, D\'avid Kunszenti-Kov\'acs and P\'eter L. Simon}

\address{A.B\'atkai, E\"otv\"os Lor\'and University, Institute of Mathematics and
Numerical Analysis and Large Networks Research Group\\ Hungarian Academy of Sciences\\
P\'azm\'any P. s\'et\'any 1/C, 1117 Budapest, Hungary.}
\email{batka@cs.elte.hu}

\address{\'A.Havasi, E\"otv\"os Lor\'and University, Institute of Mathematics and
Numerical Analysis and Large Networks Research Group\\ Hungarian Academy of Sciences\\
P\'azm\'any P. s\'et\'any 1/C, 1117 Budapest, Hungary.}
\email{hagi@nimbus.elte.hu}

\address{R.Horv\'ath, Budapest University of Technology, Institute of Mathematics and
Numerical Analysis and Large Networks Research Group\\ Hungarian Academy of Sciences\\
P\'azm\'any P. s\'et\'any 1/C, 1117 Budapest, Hungary.}
\email{rhorvath@math.bme.hu}

\address{D.Kunszenti-Kov\'acs, E\"otv\"os Lor\'and University, Institute of Mathematics and
Numerical Analysis and Large Networks Research Group\\ Hungarian Academy of Sciences\\
P\'azm\'any P. s\'et\'any 1/C, 1117 Budapest, Hungary.}
\email{daku@fa.uni-tuebingen.de}

\address{P.L.Simon, E\"otv\"os Lor\'and University, Institute of Mathematics and
Numerical Analysis and Large Networks Research Group\\ Hungarian Academy of Sciences\\
P\'azm\'any P. s\'et\'any 1/C, 1117 Budapest, Hungary.}
\email{simonp@cs.elte.hu}

\thanks{Supported by the OTKA grant Nr. K81403 and by the European Research Council Advanced Researcher Grant No. 227701.}


\date{24.03.2014}

\begin{abstract}
A large system of ordinary differential equations is approximated by a parabolic partial differential equation with dynamic boundary condition and a different one with Robin boundary condition. Using the theory of differential operators with Wentzell boundary conditions and similar theories, we give estimates on the order of approximation. The theory is demonstrated on a voter model where the Fourier method applied to the PDE is of great advantage.
\end{abstract}
\allowdisplaybreaks
\maketitle


\section{Introduction}

It has been known for a long time that there is a wonderful interplay between discrete time stochastic
processes and partial differential equations, see the seminal paper by Courant, Friedrichs and
Lewy \cite{CFL}. Since then, the pioneering work of Feller revealed deep connections to second
order differential equations with ``complicated'' boundary conditions, see the monograph by Mandl
\cite{mandl} for further details.

Our intention with this work is to go back to the roots and explore the connections of large
systems of ordinary differential equations to parabolic partial differential equations with
various (Wentzell, Robin) boundary conditions from a rather particular point of view: given a large
system of ordinary differential equations, we construct an ``approximating'' partial
differential equation, give estimates on the accuracy of this approximation, and show that in some
cases it is much easier to handle the parabolic equation than the large ODE system.

The main motivation of this theoretical investigation is to approximate a dynamic process on a
network by a partial differential equation. The network is given by an undirected graph and the
process is specified by the possible states of the nodes and the transition probabilities. Typical
examples are epidemic processes and opinion propagation on networks.  Analysing the mean field approximation
for the expected number of infected nodes in an epidemic process on a large network we were led to
a first order PDE approximation in our previous work B\'atkai et al.
\cite{Batkai-Kiss-Sikolya-Simon}. In a recent paper, in which regular random, Erd\H
os-R\'enyi, bimodal random, and Barab\'asi-Albert graphs are studied, we have shown that a suitable
choice of the coefficients in the master equation leads to an ODE approximation with tridiagonal
transition rate matrices (see Nagy, Kiss and Simon \cite{NagySimon}). Our study in this paper aims at
approximating dynamic processes on networks, for which the transition rate matrix of the underlying
Markov chain has a tridiagonal structure (with a possible extension to similar matrices).

The paper is organized as follows. First, we introduce our general notation and setup along with a
standard heuristic derivation of an approximating PDE with dynamic boundary conditions via finite differences. It is followed by a different finite difference approximation to yield an approximating PDE with Robin boundary conditions.
Then in Section 4 we present the operator semigroup theoretic setup with the general approximation theorems needed, and we show how to use the well-developed operator matrix approach to differential operators
with Wentzell boundary conditions due to Engel and coauthors \cite{Engel,Engel1, Batkai-Engel} to
 prove error estimates. Finally, in the last section we illustrate our results
with two examples:  The first one is the propagation of two opinions along a cycle graph, called a
voter-like model, the second is an $SIS$ type epidemic propagation on a complete graph.

\section{Dynamic boundary conditions}\label{sec:dyn}
\noindent
In this section we fix our notation, collect the main definitions, derive the first approximating partial differential equation, and give the main heuristics which lie behind our approximation.

Let $N\in\NN$ be a large, fixed integer, and $a, b$ and $c$ real-valued functions on $[-\frac{1}{N},1+\frac{1}{N}]$.
For $0\leq k\leq N$, let $a_k:=a(\frac{k}{N})$, $b_k:=b(\frac{k}{N})$, and $c_k:=c(\frac{k}{N})$.
Consider the following tridiagonal matrix
\[
A_N:=\left(\begin{array}{ccccccc}
b_0 & c_1& 0&\cdots&0&0&0\\
a_0 & b_1&c_2&\cdots&0&0&0\\
0&a_1&b_2&&0 &0&0\\
\vdots& & &\ddots& & & \vdots\\
0& 0&0& &b_{N-2}&c_{N-1}&0\\
0& 0&0& &a_{N-2}&b_{N-1}&c_N\\
0& 0&0 &\cdots&0 &a_{N-1}&b_N\\
\end{array}\right)
\]
and the corresponding (ODE) system
\begin{equation}\label{eq:ode}
\left\{
\begin{aligned}
\dot{x}(t)&=A_Nx(t)\\
x(0)&=v \in\CC^{N+1}
\end{aligned}
\right.
\end{equation}
 on $\CC^{N+1}$.

We wish to approximate the solution $x(t)$ to this (ODE) by considering it as a discretisation of a continuous function $u(t,z)$ on the interval $[0,1]$, i.e.,
\[
u\left(t,\frac{k}{N}\right)=x_k(t)
\]
for $0\leq k\leq N$. Now we derive an approximate (PDE) for the function $u(\cdot,\cdot)$ using the (ODE) given above.
For any $1\leq k\leq N-1$ we have:
\begin{eqnarray*}
\partial_t u\left(t,\frac{k}{N}\right)&=&\dot x_k(t)= a_{k-1} x_{k-1}(t) + b_k x_k(t) + c_{k+1} x_{k+1}(t)\\
&=& \frac{1}{2}a_{k-1}\left(x_{k-1}(t) -2x_k(t)+x_{k+1}(t)\right)\\
&&- a_{k-1}\left(\frac{x_{k+1}(t)-x_{k-1}(t)}{2}\right)\\
&&+(a_{k-1}+b_k+c_{k+1})x_k(t)\\
&&+ \frac{1}{2}c_{k+1}\left(x_{k-1}(t) -2x_k(t)+x_{k+1}(t)\right) \\
&&+ c_{k+1}\left(\frac{x_{k+1}(t)-x_{k-1}(t)}{2}\right)\\
&=& \frac{1}{2}a_{k-1}\left(u\left(t,\frac{k-1}{N}\right) -2u\left(t,\frac{k}{N}\right)+u\left(t,\frac{k+1}{N}\right)\right)\\
&&- a_{k-1}\left(\frac{u\left(t,\frac{k+1}{N}\right)-u\left(t,\frac{k-1}{N}\right)}{2}\right)\\
&&+(a_{k-1}+b_k+c_{k+1})u\left(t,\frac{k}{N}\right)\\
&&+ \frac{1}{2}c_{k+1}\left(u\left(t,\frac{k-1}{N}\right) -2u\left(t,\frac{k}{N}\right)+u\left(t,\frac{k+1}{N}\right)\right)\\
&&+ c_{k+1}\left(\frac{u\left(t,\frac{k+1}{N}\right)-u\left(t,\frac{k-1}{N}\right)}{2}\right).
\end{eqnarray*}
By considering the approximations
\[
u\left(t,\frac{k-1}{N}\right) -2u\left(t,\frac{k}{N}\right)+u\left(t,\frac{k+1}{N}\right)= \frac{1}{N^2}\left(\partial_{zz}u\left(t,\frac{k}{N}\right)+O\left(\frac{1}{N^2}\right)\right)
\]
and
\[
\frac{u\left(t,\frac{k+1}{N}\right)-u\left(t,\frac{k-1}{N}\right)}{2}=\frac{1}{N}\left(\partial_z u\left(t,\frac{k}{N}\right)+O\left(\frac{1}{N^2}\right)\right),
\]
using the functions $a, b$ and $c$, and writing $h:=\frac{1}{N}$, we obtain the approximate (PDE)
\begin{equation}\label{eq:pde_main}
\left\{
\begin{aligned}
\partial_t u(t,z)&\simeq \frac{h^2}{2}\left(a\left(z-h\right)+c\left(z+h\right)\right) \partial_{zz} u(t,z)\\
&+h(c(z+h)-a(z-h))\partial_z u(t,z)\\
& +(a(z-h)+b(z)+c(z+h)) u(t,z),
\end{aligned}
\right.
\end{equation}
valid for $z\in(0,1)$. Note that the approximation is of order $h^3$.

On the boundary, similar transformations yield the first order boundary equations
\begin{equation}\label{eq:pde_boundary_left}
\partial_{t} u(t,0)\simeq hc(h)\partial_z u(t,0) + (c(h)+b(0))u(t,0)
\end{equation}
and
\begin{equation}\label{eq:pde_boundary_right}
\partial_t u(t,1)\simeq -ha(1-h)\partial_z u(t,1) + (a(1-h)+b(1))u(t,1).
\end{equation}
Note that here the approximations are only of order $h^2$.

The initial condition $u(0,z)$ is to be chosen as a suitable interpolation of the values $x_k(0)=v_k$ at $z=\frac{k}{N}$ ($0\leq k\leq N$).

\section{Robin boundary condition}

Motivated by stochastic processes, we restrict ourselves here to the important special case where the column sums of the matrix $A_N$ are zero, i.e., $a_0=-b_0,\ b_k=-(a_k+c_k), k=1,2\ldots,N-1$, and $c_N=-b_N$.
Our aim is to find a PDE with suitable boundary
condition the appropriate discretisation of which results in
(\ref{eq:ode}), and which preserves the integral of the initial function.

Let us seek the PDE in the form
\begin{equation}\label{eq:pde}
\partial_tu(t,z)=\partial_{zz}(\alpha(z)u(t,z))+\partial_z(\beta(z)u(t,z)),
\end{equation}
where $z\in (-\frac{1}{2N},1+\frac{1}{2N})$ and
$t\in (0,T]$, and the functions $\alpha$ and $\beta$ are to be
defined. For the derivation of the boundary conditions we take
into account the requirement that
\[\int_{-\frac{1}{2N}}^{1+\frac{1}{2N}}u(t,z)\mathrm{d}z=const.\ \ \forall t\in [0,T].\]
Integrating \eqref{eq:pde} on $[-\frac{1}{2N},1+\frac{1}{2N}]$ we
obtain the equality
\begin{align*}
0=\partial_t\left(\int_{-\frac{1}{2N}}^{1+\frac{1}{2N}}u(t,z)\mathrm{d}z\right)&=
\partial_z(\alpha u)\left(1+\frac{1}{2N},t\right)-\partial_z(\alpha u)\left(-\frac{1}{2N},t\right)\\
&\quad+(\beta u)\left(1+\frac{1}{2N},t\right)-(\beta u)\left(-\frac{1}{2N},t\right),
\end{align*}
which obviously holds if
\begin{equation}\label{eq:bc1}
\partial_z(\alpha u)\left(-\frac{1}{2N},t\right)+(\beta u)\left(-\frac{1}{2N},t\right)=0, \text{ and }
\end{equation}
\begin{equation}\label{eq:bc2}
\partial_z(\alpha u)\left(1+\frac{1}{2N},t\right)+(\beta u)\left(1+\frac{1}{2N},t\right)=0
\end{equation}
hold. Consider now the continuous problem (\ref{eq:pde}) with  boundary
conditions (\ref{eq:bc1})-(\ref{eq:bc2}) and an initial condition
$u(0,z)$ obtained from a suitable interpolation of $v$ in
(\ref{eq:ode}).

Denote the approximation of the solution at the point $z=kh$ by $x_k(t),
k=0,1\ldots, N$. We seek the functions $\alpha$ and $\beta$ such
that by approximating appropriately the derivatives w.r.t. the variable $z$ in
(\ref{eq:pde}), for the functions $x_0(t),x_1(t),\ldots,$ $x_N(t)$
we obtain a system of ODE's of the form (\ref{eq:ode}).

Let us approximate the partial derivatives w.r.t. $z$ for the mesh
points of the indices $k=0,1,2,\ldots, N$ by central
differences. To this aim we define two virtual mesh points:
$-\frac{1}{N}$ and $1+\frac{1}{N}$, where the corresponding
solutions will be denoted by $x_{-1}(t)$ and $x_{N+1}(t)$,
respectively. Then
\begin{equation} \label{eq:u_k}
x_k'(t)=\frac{\alpha_{k-1}x_{k-1}-2\alpha_k x_k+\alpha_{k+1}x_{k+1}}{h^2} + \frac{\beta_{k+1}x_{k+1}-\beta_{k-1}x_{k-1}}{2h}
\end{equation}
for $k=0,1,2,\ldots, N$. Eliminate $x_{-1}$ in the equation for
$k=0$ by considering the left-hand side boundary condition
(\ref{eq:bc1}). To do so, we approximate the derivative w.r.t. $z$
by central difference, while the function value by the arithmetic
mean of the two neighboring values, $x_{-1}$ and $x_1$, to obtain
\[
\frac{\alpha_0x_0-\alpha_{-1}x_{-1}}{h}+\frac{\beta_0x_0+\beta_{-1}x_{-1}}{2}=0.
\]
From this we have
\[
\frac{\alpha_0x_0}{h^2}+\frac{\beta_0x_0}{2h}=\frac{\alpha_{-1}x_{-1}}{h^2}-\frac{\beta_{-1}x_{-1}}{2h},
\]
which yields
\begin{equation} \label{eq:u0}
x_0'(t)=\left(-\frac{\alpha_0}{h^2}+\frac{\beta_0}{2h}\right)x_0+ \left(\frac{\alpha_1}{h^2}+\frac{\beta_1}{2h}\right)x_1.
\end{equation}
Comparing \eqref{eq:u0} to the first equation of \eqref{eq:ode}, we have
\begin{equation*}
b_0=-\frac{\alpha_0}{h^2}+\frac{\beta_0}{2h} \text{ and } c_1=\frac{\alpha_1}{h^2}+\frac{\beta_1}{2h}.
\end{equation*}
Comparing the further  equations of (\ref{eq:u_k}) with system (\ref{eq:ode}),
we obtain the relations
\begin{equation} \label{ak_ck}
a_k=\frac{\alpha_k}{h^2}-\frac{\beta_k}{2h},\quad c_k=\frac{\alpha_k}{h^2}+\frac{\beta_k}{2h}.
\end{equation}
It is easy to see that $a_0=-b_0,\ b_k=-(a_k+c_k), k=1,2\ldots,N-1$, and
similar considerations on the right boundary show that $c_N=-b_N.$
The functions $\alpha$ and $\beta$ can be determined from the
equations \eqref{ak_ck} to obtain
\[
\alpha_k=\frac{(a_k+c_k)h^2}{2},\quad \beta_k={(c_k-a_k)h},
\]
from which
\begin{equation*}
\alpha(z)=\frac{(a(z)+c(z))h^2}{2} \text{ and } \beta(z)=(c(z)-a(z))h
\end{equation*}
follows.

The order of approximation of this scheme is yet to be calculated.
For $k=1,2,\ldots N$ the approximation is obviously of order $h^3$ as in the previous example.
However, in the points $z=0$ and $z=1$ (corresponding to $k=0$ and
$k=N+1$) it has only order of $h^2$, since for $z=0$ we have
\begin{align*}
&\left(-{\alpha_0}+\frac{\beta_0}{2} \right)x_0(t)+ \left({\alpha_1}+\frac{\beta_1}{2} \right)u_1(t)\\
&=\frac{\alpha_{-1}u_{-1}(t)-2\alpha_0u_0(t)+\alpha_1u_1(t)}{h^2} h^2 +\frac{\beta_1u_1(t)-\beta_{-1}u_{-1}(t)}{2h} h\\
&\qquad +\frac{1}{h}\left ( \frac{\alpha_0u_0(t)-\alpha_{-1}u_{-1}(t)}{h} h^2 + \frac{\beta_0u_0(t)+\beta_{-1}u_{-1}(t)}{2} h\right )\\
&=\partial_{zz}(\alpha u)(0,t)+O(h^4)+\partial_z(\beta u)(0,t)+O(h^3)+\frac{1}{h}(\partial_z(\alpha u)(-h/2,t)+O(h^4)\\
&\qquad+(\beta u)(-h/2,t)+O(h^3))\\
&=\partial_{zz}(\alpha u)(0,t)+\partial_z(\beta u)(0,t)+O(h^3)+\frac{1}{h}(0+O(h^3))\\
&=\partial_{zz}(\alpha u)(0,t)+\partial_z(\beta u)(0,t)+O(h^2).
\end{align*}
For $z=1$ similar relations hold.

In the following we consider the exact PDE and its solution, the latter being the approximation of the exact solution to the ODE at the points $\tfrac{k}{N}$, and show estimates on how good this approximation is.

\section{Theorems}
\noindent
Now we give a rather general setup to prove the desired estimates on the approximation. We use the theory of operator semigroups and our general reference is Engel and Nagel \cite{EN:00}  or B\'atkai et al. \cite{Batkai-Csomos-Farkas-Ostermann}.

\begin{assumptions}\label{c:apro1.ass:approx_space}
Let $X_n$, $X$ be Banach spaces and assume that there are bounded linear operators $P_n:X\to X_n$, $J_n:X_n\to X$ with the following properties:
\begin{itemize}
\item There is a constant $K>0$ with $\|P_n\|,\, \|J_n\|\leq K$ for all $n\in\NN$,
\item $ P_n J_n = I_n$, the identity operator on  $X_n$, and
\item $J_nP_n f\to f$  as  $n\to\infty$  for all  $f\in X$.
\end{itemize}
\end{assumptions}

\begin{assumptions}\label{c:apro1.ass:approx_gener}
Suppose that the operators $A_n$, $A$ generate strongly continuous semigroups on $X_n$ and $X$, respectively, and that there are constants $M\geq 0$, $\omega\in\RR$ such that the stability condition
\begin{equation}\label{c:apro1.eq:stability}
\|T_n(t)\|\leq M\ee^{\omega t} \qquad \text{ holds for all } n\in\NN,\, t\geq 0.
\end{equation}
\end{assumptions}

We will make use of a special variant of the Trotter-Kato theorem, which we cite here for convenience, see the lectures by B\'atkai, Csom{\'o}s, Farkas and Ostermann \cite[Proposition 3.8]{Batkai-Csomos-Farkas-Ostermann}.

\begin{proposition}\label{prop:appr_first_gen}
Suppose that Assumptions \ref{c:apro1.ass:approx_space} hold, that there is a dense subset $Y\subset D(A)$  invariant under the semigroup $T$ such that $P_nY\subset \dom(A_n)$, and that $Y$ is a Banach space with some norm $\|\cdot\|_Y$ satisfying
\begin{equation*}
\|T(t)\|_Y \leq M \ee^{\omega t}.
\end{equation*}
If there are constants $C>0$ and $p\in \NN$ with the property that for all $f\in Y$
\begin{equation*}
\|A_nP_n f - P_nAf\|_{X_n}\leq C\frac{\|f\|_Y}{n^p},
\end{equation*}
then for each $t>0$ there is $C'>0$ such that
\begin{equation*}
\|T_n(t)P_n f - P_nT(t)f\|_{X_n}\leq C'\frac{\|f\|_Y}{n^p}.
\end{equation*}
Moreover,  this convergence is uniform in $t$ on compact intervals.
\end{proposition}

This result can be slightly improved in case analytic semigroups are involved.

\begin{lemma}\label{lem:anal_approx}
Suppose that the conditions of Proposition \ref{prop:appr_first_gen} are satisfied and that $A$ generates an analytic semigroup. If there is $\varepsilon\in (0,1)$ and there are spaces $Y\hookrightarrow \dom(A)\hookrightarrow Z\hookrightarrow X$ such that $T(s)Z\subset Y$ for all $s>0$ and
\begin{equation*}
\|T(s)\|_{\mathcal{L}(Z,Y)}\leq \frac{M}{s^{1-\varepsilon}}
\end{equation*}
holds, then
\begin{equation*}
\|T_n(t)P_n f - P_nT(t)f\|_{X_n}\leq C'\frac{\|f\|_Z}{n^p}
\end{equation*}
for all $n\in \NN$ and $f\in Z$.
\end{lemma}

Note that this condition is for example satisfied if there is $\alpha\in (0,1)$ so that $Y=\dom((I-A)^{1+\alpha})$ and $\dom(A)\hookrightarrow Z\hookrightarrow \dom((I-A)^{\alpha+\varepsilon})$ holds.

\begin{proof}
As in the proof of B\'atkai, Kiss, Sikolya and Simon \cite[Lemma 5]{Batkai-Kiss-Sikolya-Simon}, we have the representation
\begin{equation*}
\left(P_nT(t)-T_n(t)P_n\right)f = \int_0^t T_n(t-s)\left(P_nA-A_nP_n\right)T(s)f\dd s
\end{equation*}
for all $f\in \dom(A)$. Hence, using the analyticy of the semigroup $T$, we obtain the norm estimate
\begin{multline*}
\left\|P_nT(t)f-T_n(t)P_n f\right\| \leq \int_0^t M\ee^{\omega (t-s)}\|(P_nA-A_nP_n)T(s)f\| \dd s \\
\leq \int_0^t M' C\frac{\|T(s)f\|_Y}{n^p}\dd s \leq M''\frac{\|f\|_Z}{n^p}\int_0^t \frac{1}{s^{1-\varepsilon}}\dd s,
\end{multline*}
where the constants $M'$, $C'$ and $M''$ only depend on $t>0$.
\end{proof}

We have seen in the calculations of the previous sections that our approximation is of third order in the interior of the interval and of second order on the boundary. Let us formalize now the calculations and put them into the general framework presented above.

\subsection{Dynamic boundary condition}

Our aim is now to show that for sufficiently smooth initial values the derived partial differential equation \eqref{eq:pde_main} with dynamic boundary conditions \eqref{eq:pde_boundary_left} and \eqref{eq:pde_boundary_right} is the right approximation to the ordinary differential equation \eqref{eq:ode}.

As a first step, we have to associate to the partial differential equation \eqref{eq:pde_main} with boundary conditions \eqref{eq:pde_boundary_left} and \eqref{eq:pde_boundary_right} a Banach space $X$ and a generator $A$. Following the approach of Engel \cite{Engel, Engel1} or B\'atkai and Engel \cite{Batkai-Engel}, we introduce the spaces
\begin{equation*}
X:=  C[0,1],
\end{equation*}
and
\[
\widetilde{X}:=\left\{\left(
\begin{smallmatrix}
f\\y
\end{smallmatrix}\right)
\in X\times\CC^2
\left|y=(f(0),f(1))^T\right.\right\}.
\]
%
Let us also consider the operators
\begin{multline*}
(D_m f)(z):= \frac{h^2}{2}\left(a\left(z-h\right)+c\left(z+h\right)\right) f''(z) \\ +h(c(z+h)-a(z-h))f'(z) +(a(z-h)+b(z)+c(z+h)) f(z)
\end{multline*}
defined on its maximal possible domain, $\dom(D_m):=C^{2}[0,1]$, and
\begin{equation*}
Bf:=\left(\begin{smallmatrix} hc(h) f'(0) + (c(h)+b(0))f(0) \\ -ha(1-h)f'(1) + (a(1-h)+b(1))f(1) \end{smallmatrix}\right)
\end{equation*}
defined on $\dom(D_m)$ and mapping to $\CC^2$. Our associated operator should be
\begin{equation*}
A f = D_m f  \quad \text{ with } \dom(A) :=\left\{f\in \dom(D_m)\,:\, \left(\begin{smallmatrix} D_mf(0), D_mf(1)\end{smallmatrix} \right)^T=Bf \right\}.
\end{equation*}
Note that the operator
\begin{equation*}
\tilde A= \left.\begin{pmatrix} A & 0 \\ B & 0 \end{pmatrix}\right|_{\tilde X}
\end{equation*}
is similar to the operator $A$, see B\'atkai and Engel \cite{Batkai-Engel}.
Further, for a function $f\in C[0,1]$ we introduce the notation
\begin{equation*}
f_N:=(f(0),f(\tfrac{1}{N}),\ldots,f(1))^T\in\CC^{N+1}.
\end{equation*}

After all these preparations, we can state the main result of this Section.

\begin{theorem}\label{thm:main}
Consider the ordinary differential equation given by \eqref{eq:ode} and the approximating partial differential equation \eqref{eq:pde_main} with dynamic boundary conditions \eqref{eq:pde_boundary_left} and \eqref{eq:pde_boundary_right}, where $v= u_N(0)$. If there is $\varepsilon\in (0,\tfrac{1}{2})$ such that $u(0,\cdot)\in \dom((I-A)^{\frac{1}{2}+\varepsilon})$, then for all $T>0$ there is $C=C(T)>0$ such that for all $t\in (0,T]$ we get
\begin{equation}\label{eq:thm_main}
\|u_N(t,\cdot)- x(t)\|_{\infty} \leq \frac{C}{N^2}\|u(0,\cdot)\|_{\dom((I-A)^{\frac{1}{2}+\varepsilon})}.
\end{equation}
\end{theorem}
\begin{proof}

By Engel \cite{Engel} or B\'atkai and Engel \cite[Remark 4.4]{Batkai-Engel}, the operator $A$ generates an analytic semigroup of angle $\frac{\pi}{2}$ in the space $X$, and this semigroups gives the solutions of the partial differential equation \eqref{eq:pde_main} with dynamic boundary conditions \eqref{eq:pde_boundary_left} and \eqref{eq:pde_boundary_right}.

Further, we introduce the spaces
\begin{equation*}
X_N:=\CC^{N-1}\times \CC^2
\end{equation*}
and define the operators $P_N:\widetilde{X}\to X_N$ as
\begin{equation*}
P_N(f,y):=(f_N,y)\,
\end{equation*}
Clearly, these operators and spaces satisfy the conditions in Assumptions \ref{c:apro1.ass:approx_space}.

Abusing the matrix notation, define now on $X_N\simeq\CC^{N+1}$ the operator
\begin{equation*}
\tilde A_N:=\left(\begin{array}{ccccccc}
b_1&c_2&\cdots&0&0&a_0 &0\\
a_1&b_2&&0 &0&0&0\\
\vdots& & &\ddots& & & \vdots\\
0& 0& &b_{N-2}&c_{N-1} &0 &0\\
0& 0& &a_{N-2}&b_{N-1} &0 &c_N\\
 c_1&0&\cdots&0&0& b_0& 0\\
0& 0 &\cdots &0 &a_{N-1} &0&b_N\\
\end{array}\right).
\end{equation*}

Taking $\tbinom{f}{y}\in\dom(\tilde A)$ (i.e., $f\in W^{1,1}(0,1)$ and $y_1=f(0),\, y_2=f(1)$), we see that

\begin{equation*}
\tilde A_N P_N \tbinom{f}{y} =
\begin{pmatrix}
a_0f(0)+b_1 f(\tfrac{1}{N})+ c_2 f(\tfrac{2}{N})\\
a_1 f(\tfrac{1}{N})+ b_2 f(\tfrac{2}{N})+c_3 f(\tfrac{3}{N})\\
\vdots \\
a_{N-2} f(\tfrac{N-2}{N})+ b_{N-1} f(\tfrac{N-1}{N})+c_N f(1)\\
b_0f(0)+c_1 f(\tfrac{1}{N})\\
a_{N-1} f(\tfrac{N-1}{N})+b_N f(1)
\end{pmatrix}
\end{equation*}
and that
\begin{align*}
(P_N \tilde A\tbinom{f}{y})_k &=
\tfrac{1}{2N^2} (a(\tfrac{k-1}{N})f(\tfrac{k-1}{N})-c(\tfrac{k+1}{N})f(\tfrac{k+1}{N}))f''(\tfrac{k}{N})  \\ &+\tfrac{1}{N}(a(\tfrac{k+1}{N})f(\tfrac{k+1}{N})-c(\tfrac{k-1}{N})f(\tfrac{k-1}{N}))f'(\tfrac{k}{N}) \\
&+ (a(\tfrac{k-1}{N})+b(\tfrac{k}{N})+c(\tfrac{k+1}{N})f(\tfrac{k}{N})
\end{align*}
for $k=1,2,\ldots,N-1$, and
\begin{align*}
(P_N \tilde A\tbinom{f}{y})_N =& \tfrac{1}{N}c(\tfrac{1}{N})f'(0)+ (c(\tfrac{1}{N})+ b(0))f(0)\\
(P_N \tilde A\tbinom{f}{y})_{N+1} =& -\tfrac{1}{N}c(\tfrac{N-1}{N})f'(1)+ (a(\tfrac{N-1}{N})+ b(1))f(1)
\end{align*}

By the calculations in the previous Section \ref{sec:dyn}, we see that there is $C>0$ such that
\begin{align*}
|(P_N \tilde A\tbinom{f}{y})_k - (\tilde A_NP_N \tbinom{f}{y})k| \leq& \frac{C}{N^3}\|f'''\|_{\infty},\\
|(P_N \tilde A\tbinom{f}{y})_N - (\tilde A_NP_N \tbinom{f}{y})N| \leq& \frac{C}{N^2}\|f''\|_{\infty},\\
|(P_N \tilde A\tbinom{f}{y})_{N+1} - (\tilde A_NP_N \tbinom{f}{y})_{N+1}| \leq& \frac{C}{N^2}\|f''\|_{\infty}\\
\end{align*}
hold. Since $A$ generates an analytic semigroup, it leaves $Y=C^3[0,1]$ invariant. Hence, Proposition \ref{prop:appr_first_gen} is applicable with $Y=C^3[0,1]$ and we obtain the desired estimate for all $u(0,\cdot)\in Y$. To improve this result and relax the regularity assumption on the initial value, we use the analyticity of the semigroup and Lemma \ref{lem:anal_approx} with $\alpha=\tfrac{1}{2}$.

Introducing the notation $B=I-A$, our aim now is to show that $\dom(B^{3/2})\subset C^2[0,1]\cap C^3(0,1)$. Since $\dom(B)\subset C^2[0,1]$, it is enough to show that $\dom(B^{1/2})\subset C^1(0,1)$.
Let $f\in \dom(B^{1/2})$ such that $g=B^{1/2}f$. Then by Engel and Nagel \cite[Corollary II.5.28]{EN:00},
\begin{equation*}
f=\int_0^\infty \frac{1}{\sqrt\lambda} R(\lambda+1,A)g \dd\lambda.
\end{equation*}
Further, by checking the explicit representation of the resolvent as in the proof in Engel and Nagel \cite[Theorem VI.4.5]{EN:00}, we see that the resolvent is given by the combination of exponential terms and a convolution term. Since we are in the interior of the domain, we can drop the exponential terms because they do not disturb regularity and concentrate on the convolution term. Hence we may assume that
\begin{equation*}
f=\int_0^\infty \frac{1}{\sqrt\lambda}\frac{1}{2\sqrt{\lambda+1}}\int_0^1 e^{-\sqrt{\lambda+1}|\cdot-s|}g(s) \dd s\dd\lambda.
\end{equation*}
Rewriting, we obtain
\begin{align*}
f(x)&=\int_0^\infty \frac{1}{2\sqrt\lambda\sqrt{\lambda+1}}\left\{\int_0^x e^{-\sqrt{\lambda+1}(x-s)}g(s) \dd s\right. \\
&\qquad\qquad + \left.\int_x^1 e^{-\sqrt{\lambda+1}(s-x)}g(s) \dd s\right\}\dd\lambda\\
&=\int_0^\infty \frac{1}{2\sqrt\lambda\sqrt{\lambda+1}}\left\{e^{-\sqrt{\lambda+1}x}\int_0^x e^{\sqrt{\lambda+1}s}g(s) \dd s \right. \\ &\qquad\qquad + \left.e^{\sqrt{\lambda+1}x} \int_x^1 e^{-\sqrt{\lambda+1}s}g(s) \dd s\right\}\dd\lambda.
\end{align*}
Formally differentiating with respect to $x$ behind the first integral, we obtain
\begin{eqnarray*}
\int_0^\infty \frac{1}{2\sqrt\lambda\sqrt{\lambda+1}}
&&\left\{\frac{-1}{\sqrt{\lambda+1}}e^{-\sqrt{\lambda+1}x}\int_0^x e^{\sqrt{\lambda+1}s}g(s) \dd s+g(x)\right.\\
&&+\left.\frac{1}{\sqrt{\lambda+1}}e^{\sqrt{\lambda+1}x} \int_x^1 e^{-\sqrt{\lambda+1}s}g(s) \dd s-g(x)\right\}\dd\lambda\\
=\int_0^\infty\frac{1}{2\sqrt{\lambda(\lambda+1)}}&&
\left\{\frac{-1}{\sqrt{\lambda+1}}e^{-\sqrt{\lambda+1}x}\int_0^x e^{\sqrt{\lambda+1}s}g(s) \dd s\right.\\
&&+\left.\frac{1}{\sqrt{\lambda+1}}e^{\sqrt{\lambda+1}x} \int_x^1 e^{-\sqrt{\lambda+1}s}g(s) \dd s\right\}\dd\lambda.
\end{eqnarray*}
Since this improper integral converges uniformly in $x$ on any closed subinterval of $(0,1)$, and depends continuously on $x$, the function $f$ is indeed continuously differentiable on $(0,1)$.
\end{proof}

\subsection{Robin boundary condition}

Using analogous argument, we can prove the approximating property of the PDE with Robin boundary conditions. To this end, we introduce the space $X=C[-\frac{h}{2},1+\frac{h}{2}]$ and the operator

\begin{equation*}
Af(z):=\frac{\dd^2}{\dd z^2}\left(h^2\frac{a(z)+c(z)}{2}f(z)\right)+\frac{\dd}{\dd z}\left(h\frac{c(z)-a(z)}{2}f(z)\right),
\end{equation*}
with domain
\begin{multline*}
D(A):=\big\{f\in C^1[-\tfrac{h}{2},1+\tfrac{h}{2}]\cap C^2(-\tfrac{h}{2},1+\tfrac{h}{2})\,:\\
 \frac{\dd}{\dd z}\big(h^2\frac{a(z)+c(z)}{2}f(z)\big)+\big(h\frac{c(z)-a(z)}{2}f(z)\big) = 0 \\
 \text{ for } z=-\tfrac{h}2, 1+\tfrac{h}{2}\big\}.
\end{multline*}
Further, as before, for a function $f\in C[0,1]$ we use the notation
\begin{equation*}
f_N:=(f(0),f(\tfrac{1}{N}),\ldots,f(1))^T\in\CC^{N+1}.
\end{equation*}

\begin{theorem}\label{thm:main2}
Consider the ordinary differential equation given by \eqref{eq:ode} and the approximating partial differential equation \eqref{eq:pde} with Robin-type boundary conditions \eqref{eq:bc1} and \eqref{eq:bc2}, where $v= u_N(0)$. If there is $\varepsilon\in (0,\tfrac{1}{2})$ such that $u(0,\cdot)\in \dom((I-A)^{\frac{1}{2}+\varepsilon})$, then for all $T>0$ there is $C=C(T)>0$ such that for all $t\in (0,T]$ we get
\begin{equation}\label{eq:thm_main2}
\|u_N(t,\cdot)- x(t)\|_{\infty} \leq \frac{C}{N^2}\|u(0,\cdot)\|_{\dom((I-A)^{\frac{1}{2}+\varepsilon})}.
\end{equation}
\end{theorem}
\begin{proof}
The proof can be carried out in a completely analogous way as for the previous theorem. For second order differential operators with Robin-type boundary conditions we refer to the works by Arendt and coauthors \cite{AMPR,AW} or Warma \cite{W}.
\end{proof}


\section{Applications}

The main motivation of the previous theoretical investigation is to approximate a dynamic process on a network with a partial differential equation and to justify empirical observations. The network is usually given by an undirected graph and the process can be specified by the possible states of the nodes and the transition rate probabilities. The latter is the probability that the state of a node changes from one state to another depending on the states of the neighbouring nodes. In certain classes of models the complete state space can be reduced (using e.g. mean field approximations or structural symmetries), leading to tridiagonal systems.

In this section we show in two cases how the theory can be applied. The first one is the propagation of two opinions along a cycle graph, called a voter-like model, the second is an $SIS$ type epidemic propagation on a complete graph. As usual in the literature, in both models the natural Markov process is conditioned on not reaching the absorbing state(s).
\subsection{Voter-like model on a cycle graph}

Let us consider a cycle graph with $N+2$ nodes, i.e. we have a connected graph, in which each node has two neighbours. A node can be in one of two states, let us denote them by 0 and 1. These states represent two opinions propagating along the edges of the graph (see Holley and Liggett \cite{HolleyLiggett}). If a node is in state 0 and has $k$ neighbours in state 1 ($k=0,1,2$), then its state will change to 1 with probability $k\tau \Delta t$ in a small time interval $\Delta t$. This expresses that opinion 1 invades that node. The opposite case can also happen, that is a node in state 1 can become a node with opinion 0 with a probability $k\gamma \Delta t$ in a small time interval $\Delta t$, if it has $k$ neighbours in state 0. The parameters $\tau$ and $\gamma$ characterize the strengths of the two opinions. The model originates in physics, where in a network of interacting particles each node holds either spin 1 or -1 (see Vazquez and Eguiluz \cite{VazquezEguiluz}). In a single event, a randomly chosen node adopts the spin of one of its neighbors, also chosen at random.

Assuming that at the initial instant the territories of the two opinions are connected sets, the underlying conditioned Markov chain can be given as follows. The state space is the set $\{ 0,1, 2, \ldots , N\}$, where a number $k$ represents the state in which there are $k+1$ nodes in state 1 and they form a connected arc along the cycle graph. Starting from state $k$ the system can move either to state $k+1$ or to $k-1$, since at a given instant only one node can change its state (by using the usual assumption that the changes at the nodes can be given by independent Poisson processes). When the system moves from state $k$ to $k+1$ then a new node in state 1 appears at one of the two ends of the arc of state 1 nodes. Hence the rate of this transition is $2\tau$, expressing that a node in state 0 and having a single neighbour in state 1 becomes a state 1 node, and this can happen at both ends of the state 1 territory. Similarly, the rate of transition from state $k$ to $k-1$ is $2\gamma$. Let us denote by $x_k(t)$ the probability that the system is in state $k$. The above transition rates lead to the differential equation
\begin{equation*}
\dot x(t) = 2\tau x_{k-1} (t) - 2(\tau+\gamma) x_k(t) + 2\gamma x_{k+1}(t) .
\end{equation*}
(For $k=0$ and for $k=N$ the equations contain only two terms.) Thus our system of ODEs takes the form given in (\ref{eq:ode}) with $a\equiv 2\tau$, $c\equiv 2\gamma$ and $b|_{[1/N,1-1/N]}\equiv -2(\tau+\gamma)$, $b(0)=-2\tau$, $b(1)=-2\gamma$, yielding to the differential equation
\begin{equation}\label{eq:odeVoter}
\dot{x}(t)=A_v x(t)
\end{equation}
with the matrix
\begin{equation*}
A_v=2\left(\begin{array}{ccccccc}
-\tau & \gamma& 0&\cdots&0&0&0\\
\tau &-(\tau+\gamma)&\gamma&\cdots&0&0&0\\
0&\tau&-(\tau+\gamma)&&0 &0&0\\
\vdots& & &\ddots& & & \vdots\\
0& 0&0& &-(\tau+\gamma)&\gamma&0\\
0& 0&0& &\tau&-(\tau+\gamma)&\gamma\\
0& 0&0 &\cdots&0 &\tau&-\gamma\\
\end{array}\right)
\end{equation*}
subject to the initial condition $x(0)=v \in\CC^{N+1}$. Using (\ref{eq:pde_main}) the corresponding approximating PDE is then given by:
\begin{equation}\label{eq:pdeVoter}
\left\{\begin{aligned}
\partial_t u(t,z)&= (\tau+\gamma)h^2\partial_{zz} u(t,z)+2(\gamma-\tau)h\partial_z u(t,z)\\
\partial_t u(t,0)&= 2\gamma h\partial_z u(t,0) + 2(\gamma-\tau) u(t,0)\\
\partial_t u(t,1)&= -2\tau h\partial_z u(t,1) - 2(\gamma-\tau) u(t,1).
\end{aligned}\right.
\end{equation}

To illustrate the effectiveness of our method numerically, we consider the special case of $\tau=\gamma=\alpha/2$, leading to the simplified equations
\begin{equation}\label{eq:pdeVoter2}
\left\{
\begin{aligned}
\partial_t u(t,z)&=\alpha h^2\partial_{zz} u(t,z)\\
\partial_t u(t,0)&=\alpha h\partial_z u(t,0)\\
\partial_t u(t,1)&=-\alpha h\partial_z u(t,1),
\end{aligned}
\right.
\end{equation}
where the associated generator has all its eigenvalues in $(-\infty,0]$.
Wishing to apply the Fourier method, we look for the solution in the form
\[
u(t,z)=\sum_{j=0}^\infty c_je^{\lambda_jt}w_j(z)
\]
It is enough to find the eigenfunctions $e^{\lambda_jt}w_j(z)$. The PDE and the boundary conditions then yield the system of equations
\begin{equation*}
\left\{
\begin{aligned}
\lambda w&=\alpha h^2 w''\\
\lambda w(0)&=\alpha hw'(0)\\
\lambda w(1)&=-\alpha hw'(1).
\end{aligned}
\right.
\end{equation*}
The first equation yields
\[
w_j(z)=c_{1,\lambda_j} \cos (\omega_j z/h) +c_{2,\lambda_j} \sin(\omega_j z/h),
\]
with $\lambda_j=-\omega_j^2$, $\omega_j\geq 0$. Substituting into the first boundary condition we obtain $-\omega_j c_{1,\lambda_j}=c_{2,\lambda_j}$, allowing us to choose $c_{1,\lambda_j}=1$ and hence write
\[
w_j(z)=\cos (\omega_j z/h) -\omega_j \sin(\omega_j z/h).
\]
 Now substituting into the second boundary condition we obtain
\[
\tan\left(\frac{\omega_j}{h}\right)=\frac{2\omega_j}{\omega_j^2-1}.
\]
This has exactly one solution in each interval $\left((2j-1)h\frac{\pi}{2},(2j+1)h\frac{\pi}{2}\right)$ for $j\geq 0$.
The constants $c_j$ are determined by the initial condition
\[
u(0,z)=\sum_{j=0}^\infty c_j w_j(z).
\]

Introducing the infinite matrix $G=((\langle w_k,w_l\rangle)_{k,l})$ , where $\langle\cdot,\cdot\rangle$ is the $L^2$ scalar product, and the vectors $U=(\langle w_j,u(0,\cdot)\rangle_j)$ and $c=(c_j)_j$, this leads to the equation
\begin{equation}\label{eq:Fourier}
Gc=U
\end{equation}
for the Fourier coefficients of the solution.

\begin{figure}[h!]
\begin{center}
\includegraphics[width=10cm]{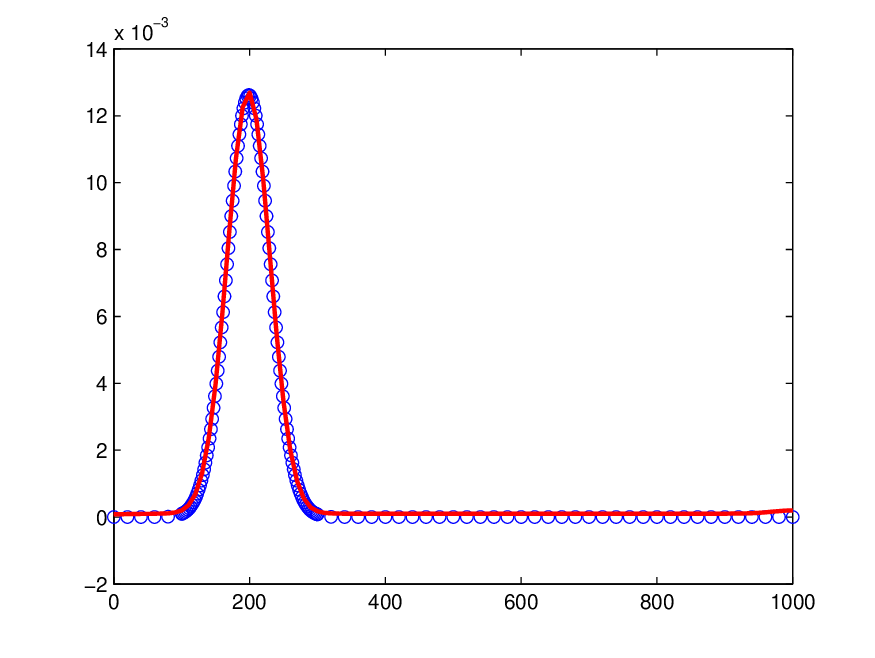}
\caption{The probability distribution $x_k(t)$, $k=0,1, 2, \ldots
, N$ at time $t=500$ obtained from system (\ref{eq:odeVoter})
(circles) and the solution $z\mapsto u(t,z)$ of the PDE
(\ref{eq:pdeVoter}) at time $t=500$ (continuous line), with
initially 200 nodes in state 1 with probability 1, and with
$N=1000$, $\tau=0.5$, $\gamma=0.5$.} \label{fig1}
\end{center}
\end{figure}

In Figure \ref{fig1} the solution of system (\ref{eq:odeVoter}) is compared to the
solution of the PDE (\ref{eq:pdeVoter}) when $\tau=\gamma=0.5$, the latter was
plotted using the Fourier method with the first 40 eigenfunctions. The first 40
eigenvalues were determined by using Newton's method within each interval given
above, and then we solved equation (\ref{eq:Fourier}) restricted to the first 40
variables. We observed that on our desktop computer MATLAB needed 15.719000 seconds
to get the ODE solution at $t=100$,
while for the Fourier method 0.016000 seconds were needed to solve the PDE.

We also compared the solutions of the ODE and the PDE for the
Robin-type boundary condition. For the Voter-like model equation
\eqref{eq:pde} has the form
\begin{equation}\label{eq:Voter_pde_Robin}
\partial_tu(t,z)=(\tau+\gamma)h^2\partial_{zz}u(t,z)+2(\gamma-\tau)h\partial_zu(t,z),
\end{equation}
with $z\in \left(-\frac{1}{2N},1+\frac{1}{2N}\right), t\in (0,T]$, and the Robin-type boundary conditions read
as
\begin{eqnarray}\label{eq:leftbc_Voter}
(\tau+\gamma)h\partial_zu\left(t,-\frac{1}{2N}\right)
+2(\gamma-\tau)u\left(t,-\frac{1}{2N}\right)=0,
\\\label{eq:rightbc_Voter}
(\tau+\gamma)h\partial_zu\left(t,1+\frac{1}{2N}\right)
+2(\gamma-\tau)u\left(t,1+\frac{1}{2N}\right)=0
\end{eqnarray}
for $t\in [0,T]$. The system   (\ref{eq:odeVoter}) was solved with
MATLAB's ode45 solver, while the partial differential equation
with MATLAB's pdepe solver. The results of the comparison are
shown in Fig.~\ref{abra_Robin} at time $t=500$ for two different
parameter choices.
\begin{figure} [h!]
\begin{center}
\mbox{\epsfig{file=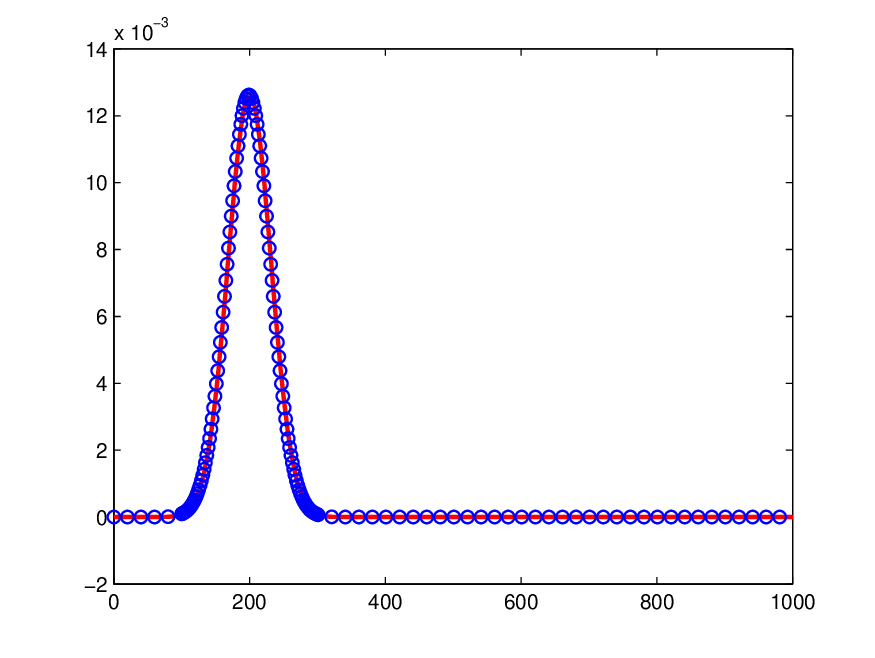,height=7cm,width=0.5\textwidth}\epsfig{file=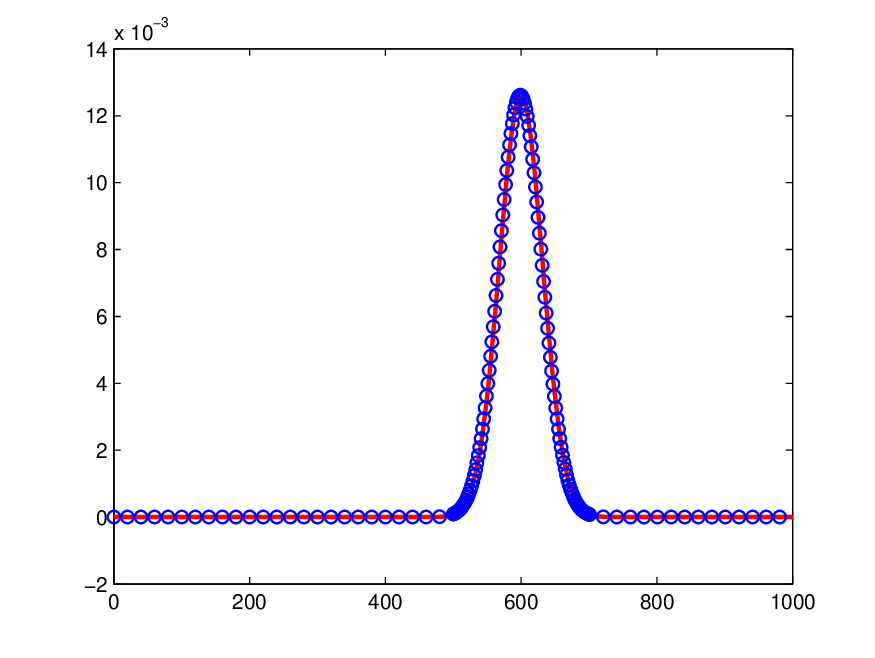,height=7cm,width=0.5\textwidth}}
\caption{The probability distribution $x_k(t)$, $k=0,1, 2, \ldots
, N$ at time $t=500$ obtained from system (\ref{eq:odeVoter})
(circles) and the solution $z\mapsto u(t,z)$ of the PDE
(\ref{eq:Voter_pde_Robin})  with boundary conditions
(\ref{eq:leftbc_Voter}) and (\ref{eq:rightbc_Voter}) at time
$t=500$ (continuous line)  with initially 200 nodes in state 1
with probability 1, and with $N=1000$,  for $\tau=\gamma=0.5$
(left panel) and for $\tau=0.7$, $\gamma=0.3$ (right panel).}
\label{abra_Robin}
\end{center}
\end{figure}

\subsection{$SIS$ disease transmission model on a complete graph}

The second motivation of our study comes from epidemiology where a paradigm disease transmission model is the
simple susceptible-infected-susceptible ($SIS$) model on a completely connected graph with $N+1$
nodes, i.e. all individuals are connected to each other. From the disease dynamic viewpoint, each
individual is either susceptible ($S$) or infected ($I$) -- a susceptible one with $k+1$ infected neighbours can be infected
at rate ($k\tau$) and the infected ones can recover at a given rate ($\gamma$) and become susceptible again. Since the graph is complete, the state space is the set $\{ 0,1, 2, \ldots , N\}$, where a number $k$ represents the state in which there are $k$ infected nodes. Starting from state $k$ the system can move either to state $k+1$ or to $k-1$, since at a given instant only one node can change its state. When the system moves from state $k$ to $k+1$ then a susceptible node becomes infected. Hence the rate of this transition is $k(N-k)\tau$, expressing that any of the $N-k$ susceptible nodes can become infected and each of them has $k$ infected neighbours (since the graph is complete). The rate of transition from state $k$ to $k-1$ is $k\gamma$, because any of the $k$ infected nodes can recover. Let us denote by $x_k(t)$ the probability that the system is in state $k$, i.e. there are $k$ infected nodes. The above transition rates lead to the differential equation
$$
\dot x(t) = (k-1)(N-k+1)\tau x_{k-1} (t) - (k(N-k)\tau+k\gamma) x(t) + (k+1)\gamma x_{k+1}(t) .
$$
(For $k=0$ and for $k=N$ the equations contain only two terms.) Thus our system of ODEs takes the form given in (\ref{eq:ode}) with $a_k=k(N-k)\tau$, $c_k=k\gamma$ and $b_k=-a_k-c_k$, that is $a(z)=N^2\tau z(1-z)$, $c(z)=N\gamma z$ and $b(z)=-a(z)-c(z)$. We note that an approximation of this system by a first order PDE was investigated in B{\'a}tkai, Kiss, Sikolya and Simon \cite{Batkai-Kiss-Sikolya-Simon}. According to (\ref{eq:pde_main}) our method yields the following second order approximation
\begin{eqnarray*}
\partial_t u(t,z)&=&\frac{\alpha (z-h)(1-z+h)+\gamma (z+h)}{2}h\partial_{zz} u(t,z)\\
&&+(\gamma (z+h)-\alpha (z-h)(1-z+h))\partial_z u(t,z)\\
&&+(\alpha (2z-1-h) +\gamma ) u(t,z)\\
\partial_t u(t,0)&=&\gamma h\partial_z u(t,0) + \gamma  u(t,0)\\
\partial_t u(t,1)&=&-\alpha (1-h)h\partial_z u(t,1) + \alpha (1-h) u(t,1).
\end{eqnarray*}
Our theorem implies that the solution of this PDE approximates the solution of the corresponding ODE (\ref{eq:ode}) in the order of $1/N^2$. We note that the usually used first order PDE approximates the ODE in the order of $1/N$. The advantage of that first order PDE is that it can be solved analytically yielding the well-known mean-field approximation for the expected number of infected nodes, see B\'atkai et al. \cite{Batkai-Kiss-Sikolya-Simon}. Our second order PDE cannot be solved analytically, hence only a numerical approximation can be obtained by using our method. It is the subject of future work to derive PDE approximations for epidemic propagation on different random graphs and compare their solutions to those of the original ODE system.

\end{document}